\documentclass{amsproc}
\usepackage{mathabx}

\usepackage{}
\copyrightinfo{2009}{American Mathematical Society}

\newtheorem{theorem}{Theorem}[section]

\newtheorem{corollary}[theorem]{Corollary}

\usepackage{tikz}
\usepackage{calc}
\usetikzlibrary{decorations.markings}
\tikzstyle{vertex}=[circle, draw, inner sep=0pt, minimum size=1pt]
\newcommand{\vertex}{\node[vertex]}

\theoremstyle{definition}
\newtheorem{definition}[theorem]{Definition}

\theoremstyle{remark}

\numberwithin{equation}{section}

\begin{document}

\title{Very cost effective bipartition in $\Gamma(Z_n)$}

\author{Ravindra Kumar}
\address{Department of Mathematics,IIT Patna, Bihta campus, Bihta-801 106}
\curraddr{}
\email{ravindra.pma15@iitp.ac.in}
\thanks{}

\author{Om Prakash}
\address{Department of Mathematics \\
IIT Patna, Bihta campus, Bihta-801 106}
\curraddr{}
\email{om@iitp.ac.in}
\thanks{}

\subjclass[2010]{13M99, 05C25, 05C76.}

\keywords{Very cost effective bipartition, Zero divisor graph, nilradical graph, non-nilradical graph, line graph, Total graph}

\date{}

\maketitle
\begin{abstract}

 Let $\mathbb{Z}_n$ be the finite commutative ring of residue classes $modulo~ n$ and $\Gamma(\mathbb{Z}_n)$ be its zero-divisor graph.  The nilradical graph and non-nilradical graph of $\mathbb{Z}_n$ are denoted by $N(\mathbb{Z}_n)$ and $\Omega(\mathbb{Z}_n)$ respectively. In 2012, Haynes et al. \cite{twh} introduced the concept of very cost effective graph. For a graph $G = (V,E)$ and a set of vertices $S \subseteq V$, a vertex $v \in S$ is said to be very cost effective if it is adjacent to more vertices in $V \setminus S$ than in $S$. A bipartition  $\pi = \{S, V \setminus S\}$ is called very cost effective if both $S$ and $V \setminus S$ are very cost effective sets \cite {twh, twh1}. In this paper, we investigate the very cost effective bipartition of $\Gamma(Z_n)$, where $n = p_1 p_2 \cdots p_m$,  here all $p_{i}^{'}$s are distinct primes.  In addition, we discuss the cases in which  $N(\mathbb{Z}_n)$ and  $\Omega(\mathbb{Z}_n)$ graphs have very cost effective bipartition for different $n$. Finally, we derive some results for very cost effective bipartition of the Line graph and Total graph of $\Gamma(\mathbb{Z}_n)$, denoted by $ L(\Gamma(\mathbb{Z}_n))$ and $ T(\Gamma(\mathbb{Z}_n))$ respectively.
\end{abstract}

\section{INTRODUCTION}

Let $R= \mathbb{Z}_n$ be the finite commutative ring of residue classes $modulo ~n$ with identity$(1 \neq 0)$ and $\Gamma(\mathbb{Z}_n)$ be its zero-divisor graph. The study of zero-divisor graphs of commutative rings reveals interesting relation between ring theory and graph theory. Because, algebraic tools help to understand graphs properties and vice versa. An element $z(\neq 0) \in R$ is said to be a zero-divisor if there exists a non-zero $r \in R$ such that $rz = 0$. The set of zero-divisors is denoted by $Z(R)$ and zero-divisor graph of $R$, $\Gamma(R)$, is the graph whose vertices are the zero-divisors of $R$ and its two vertices are connected by an edge if and only if their product is $0$. In 1988, the concept of zero-divisor graph of a commutative ring was introduced by I. Beck \cite{beck} in context of coloring of rings and were redefined in 1999 by D. F Anderson and P. Livingston in\cite{ander}.\\
For a graph $G = (V, E)$, the open neighborhood of a vertex $u \in V$ is the set $N(u) = \lbrace v ~ \vert ~ uv \in E \rbrace$, and the closed neighborhood  of $u$ is the set $N[u] = N(u) \cup \lbrace u \rbrace$. In the same fashion, the open neighborhood of a set $S \subseteq V$ is the set $N(S) = \cup_{u \in S} N(u)$, and the closed neighborhood is the set $N[S] = N(S) \cup S$ respectively. The order of the open neighborhood of a vertex $u \in V$ is denoted by $\vert N(u) \vert$. The degree of a vertex $v$ in a graph $G$, denoted by $deg(v)$, is $\vert N(v) \vert$. For basic definitions and results on graph we refer \cite{harary}.\\

   In 2012, the concept of cost effective and very cost effective sets in graphs were introduced by Haynes et al.\cite{twh} and further studied in \cite{twh1} for various graphs. A vertex $v$ in a set $S$ is said to be cost effective if it is adjacent to at least as many vertices in $V \setminus S$ as in $S$, that is, $\vert N(v) \cap S\vert \leq \vert N(v) \cap V \setminus S\vert$. A Vertex $v$ in a set $S$ is very cost effective if it is adjacent to more vertices in $V \setminus S$ than in $S$, that is, $\vert N(v) \cap S\vert < \vert N(v) \cap V \setminus S\vert$. A set $S$ is (very) cost effective if every vertex $v \in S$ is (very) cost effective. Moreover, very cost effective bipartition were also introduced in \cite{twh}. A bipartition $\pi = {\lbrace S,V \setminus S \rbrace}$ is called cost effective if each of $S$ and $V \setminus S$ is cost effective, and $\pi$ is very cost effective if each of $S$ and $V \setminus S$ is very cost effective. Graphs that have a (very) cost effective bipartition are called (very) cost effective graphs. It was shown in \cite{twh} that every connected, non-trivial graph is cost effective. Also, they observed in \cite{twh1} that all bipartite graphs with no isolated vertices are very cost effective.\\

   A line graph $L(G)$ of a simple graph $G$ is obtained by associating a vertex with each edge of the graph $G$ and connecting two vertices by an edge if and only if the corresponding edges of $G$ have a vertex in common. The total graph $T(G)$ of the graph $G$ has a vertex for each edge and each vertex of $G$ and an edge in  $T(G)$ for every edge-edge, vertex-edge, and vertex-vertex adjacency in $G$. Here, we illustrate an example of a Zero divisor graph and its line graph over ring $\mathbb{Z}_{16}$:\\

\[\begin{tikzpicture}
	\vertex (6) at (0,1) [label=left:$6$]{};
	\vertex (2) at (0.5,2.5) [label=left:$2$]{};
	\vertex (10) at (0.5,-0.5) [label=below:$10$]{};
	\vertex (8) at (2,1) [label=above:$8$]{};
	\vertex (4) at (3.5,2.5) [label=above:$4$]{};
	\vertex (14) at (3.5,-0.5) [label=right:$14$]{};
	\vertex (12) at (4,1) [label=right:$12$]{};
	\path
		(2) edge (8)
		(6) edge (8)
		(8) edge (10)
		(8) edge (12)
		(8) edge (14)
		(8) edge (4)
		(12) edge (4)	
	;
\end{tikzpicture}\]
\hspace{4.5cm} \textbf{Figure 1} $\varGamma(\mathbb{Z}_{16})$

\[\begin{tikzpicture}

	\vertex (c1) at (4.5,-1) [label=below:${(6,8)}$]{};
	\vertex (c0) at (6,1.5) [label=right:${(2,8)}$]{};
	\vertex (c2) at (5,4) [label=right:${(10,8)}$]{};
	\vertex (c6) at (3,5) [label=above:${(4,12)}$]{};
	\vertex (c5) at (0.5,3.5) [label=above:${(4,8)}$]{};
	\vertex (c3) at (-0.5,1) [label=left:${(14,8)}$]{};
	\vertex (c4) at (1,-1) [label=below:${(12,8)}$]{};
	\path
		(c0) edge (c1)
		(c0) edge (c2)
		(c0) edge (c3)
		(c0) edge (c4)
		(c0) edge (c5)
		(c1) edge (c2)
		(c1) edge (c3)
		(c1) edge (c4)
		(c1) edge (c5)
		(c4) edge (c2)
		(c2) edge (c3)
		(c5) edge (c2)
		(c4) edge (c3)
		(c5) edge (c4)
		(c5) edge (c3)
		(c5) edge (c6)
		(c4) edge (c6)
	;

\end{tikzpicture}\]
\hspace{4.5cm} \textbf{Figure 2} $L(\varGamma(\mathbb{Z}_{16}))$\\

\section{Bipartition in $\Gamma(\mathbb{Z}_n)$ and $L(\Gamma(\mathbb{Z}_n))$}

Let $\pi = {\lbrace R,B \rbrace}$ be a bipartition of the graph $G$. If $\pi$ is a very cost effective bipartition of $G$, then we say that $G$ is very cost effective under $\pi$.

\begin{theorem}
If $n = p_1.p_2...p_m, m \geq 1$ and $p_1 < p_2 <...< p_m$ are primes, then $\Gamma(\mathbb{Z}_n)$ is a very cost effective graph.\\

\end{theorem}

\begin{proof}

 Let $n = p_1.p_2...p_m, m \geq 1$ for distinct primes $p_1,p_2,...,p_m$. Then all the zero divisor elements of $\mathbb{Z}_n$ are $p_1, 2p_1,..., (p_2...p_m-1)p_1; p_2, 2p_2,...,(p_1p_3\\....p_m-1)p_2;..., p_m, 2p_m,...,(p_1p_2...p_{m-1}-1)p_m$. Let $\pi = {\lbrace R,B \rbrace}$ be a bipartition of vertices in $\Gamma(\mathbb{Z}_n)$. Suppose the set $R$ contains all the elements which are multiple of $p_m$ and the other set $B$ contains rest of elements. Then $B$ is an independent set because it does not contain any element which is multiple of $p_m$. Now, we take $u \in B$. Then there exists at least one element $v \in R$ such that $uv = 0$. Therefore, $\vert N(u) \cap B\vert = 0$ and $\vert N(u) \cap R\vert \geq 1$, for all $u \in B$. Hence, all the elements in set $B$ are very cost effective and thus the set $B$ is very cost effective set. \\
 Now, we divide $R$ into two sets $R_1$ and $R_2$. $R_1$ is the subset of $R$ containing multiple of $p_m$ as well as some of $p_i's$ (not all at a time), for $i = 1,...,m-1$ and $R_2$ is containing those multiples of $p_m$ which are not multiple of any  $p_i's$, $i = 1,...,m-1$. Then element $v \in R_2$ is not adjacent to any element in $R$. But, these elements are adjacent to the elements of set $B$, so $\vert N(v) \cap R\vert = 0$ and $\vert N(v) \cap B\vert \geq 1$. Again, if  $u \in R_1$, then number of vertices adjacent to $u$ in $R$ is $\lfloor \frac{p_1.p_2...p_{m-1}-1}{\Pi  p_i}\rfloor = M$, where $\Pi p_i $ is the product of those $p_i's$ $i = 1,...,m-1$ which are not available in $u$ and number of vertices which are adjacent to $u$ in $B$ is $((\Pi p_i)p_m -1-M)$ where $\Pi p_i $ is the product of those $p_i's$ $i = 1,...,m-1$ which are available in $u$. Since $\vert N(u) \cap R\vert = M$ and $\vert N(u) \cap B\vert > M$. So $\vert N(u) \cap R\vert < \vert N(u) \cap B\vert$. Therefore, the set $B$ is also very cost effective and the partition $\pi = {\lbrace R,B \rbrace}$ is very cost effective bipartition. Hence, the graph $\Gamma(\mathbb{Z}_n)$ is very cost effective graph.
\end{proof}

\begin{corollary}
$\Gamma(\mathbb{Z}_n)$ is very cost effective graph, for $n = pq$, where $p, q$ are distinct primes.

\end{corollary}

\begin{proof}

For $n = pq$, there are two independent sets of vertices in $\Gamma(\mathbb{Z}_n)$. One set contains multiple of $p$ and other multiple of $q$ respectively. Therefore, the graph $\Gamma(\mathbb{Z}_n)$ is a complete bipartite graph. Since every bipartite graph without isolated vertices is very cost effective. Thus, $\Gamma(\mathbb{Z}_n)$ is very cost effective graph.

\end{proof}

\begin{theorem}
Let $p$ and $q$ be distinct primes and $n$ a positive integer.\\
    $(i)$ If $n = p^2 q$, $p, q \geq 2$, then $\Gamma(\mathbb{Z}_n)$ is very cost effective.\\
    $(ii)$ If $n = p^2 q^2$, $p, q \geq 3$ and $p < q$, then $\Gamma(\mathbb{Z}_n)$ is very cost effective.\\

\end{theorem}

\begin{proof}

 $(i)$ Let $n = p^2 q$, where $p, q$ be distinct primes. Certainly, zero divisor elements of $\mathbb{Z}_n$ are either multiples of $p$ or multiples of $q$ or multiples of $pq$. Let $R$ be the set of vertices contains all those elements which are multiple of $q$ and $B$ contains all those elements which are multiple of $p$ but not $q$. Let $\pi = {\lbrace R,B \rbrace}$ be the bipartition of $V(\Gamma(\mathbb{Z}_n))$. Then $B$ is an independent set. Now, take $u \in B$, then $\vert N(u)\cap B \vert = 0$ and $\vert N(u)\cap R \vert \geq 1$. In the set $R$, take $v \in R$. If $v$ is multiple of $pq$, then $\vert N(v)\cap R \vert = p-2$ and $\vert N(v)\cap B \vert > p-2$. Again, if $v$ is a multiple of $q$ only, then $\vert N(v)\cap R \vert = 0$ and $\vert N(v)\cap B \vert \geq 1$. Hence, from both the condition $\vert N(v)\cap R \vert < \vert N(v)\cap B \vert $. Therefore, the sets $R$ and $B$ are very cost effective and thus, the graph $\Gamma(\mathbb{Z}_n)$ is very cost effective.\\

$(ii)$ Let $n = p^2 q^2$, where $p, q$ are distinct odd primes. Then the zero divisor elements are either multiple of $p$ or $q$ or both. Now, take a bipartition $\pi = {\lbrace R,B \rbrace}$ in such a way that set $R = R_1 \cup R_2 \cup R_3$ where $R_1 = \lbrace v \in V(\Gamma(\mathbb{Z}_{p^2 q^2})) : p^2\mid v\rbrace$, $R_2 = \lbrace v \in V(\Gamma(\mathbb{Z}_{p^2 q^2})) : p\mid v~ and~ q\notdivides v \rbrace$ and $R_3 = \lbrace v \in V(\Gamma(\mathbb{Z}_{p^2 q^2})) :pq\mid v ~and~ p^2 \notdivides v ~and~ q^2 \notdivides v \rbrace$. Here, $R_3$ contains $\frac{q(p-2)+1}{2}$ number of vertices. Now, set $B = B_1 \cup B_2 \cup B_3$ where $B_1 = \lbrace v \in V(\Gamma(\mathbb{Z}_{p^2 q^2})) : q^2\mid v \rbrace$, $B_2 = \lbrace v \in V(\Gamma(\mathbb{Z}_{p^2 q^2})) : q\mid v~ and~ p\notdivides v \rbrace$ and $ B_3 = \lbrace v \in V(\Gamma(\mathbb{Z}_{p^2 q^2})) :pq\mid v ~and ~p^2 \notdivides v ~and~ q^2 \notdivides v \rbrace$. Here, $B_3$ contains $\frac{p(q-2)+1}{2}$ number of vertices. Let $u\in R$. If $u$ is not a multiple of $q$, then $\vert   N(u)\cap R\vert = 0$ and $\vert   N(u)\cap B\vert \geq 1$. Again, if $u$ is a multiple of $p$ as well as $q$, then $\vert   N(u)\cap R\vert = \frac{pq-3}{2}$ and $\vert   N(u)\cap B\vert \geq \frac{pq-1}{2}$. So $\vert N(u) \cap R\vert < \vert N(u) \cap B\vert $ and the set $R$ is very cost effective.\\ Similarly, set $B$ is also very cost effective and the bipartition $\pi$ is very cost effective bipartition. Thus, the graph $\Gamma(\mathbb{Z}_n)$ is very cost effective.
\end{proof}

\begin{theorem}

$L(\Gamma(\mathbb{Z}_n))$ is very cost effective, Where $n = pq$, $p < q$ and $p, q$ are primes.\\

\end{theorem}

\begin{proof}

 If $n = pq$, then zero divisor graph $\Gamma(\mathbb{Z}_n)$ is a complete bipartite graph. So, there are two independent sets of vertices in which each vertex of a set is adjacent to every vertex of the other set. We draw the line graph of $\Gamma(\mathbb{Z}_n)$ with $( p-1)(q-1)$ vertices. This line graph is $(p+q-4)$ regular graph. Let $[u_i, v_j] \in V(L(\Gamma(\mathbb{Z}_n)))$, where $u_i's$ are multiple of $p$ $i.e$ $u_i = p.i$, $1 \leq i \leq q-1$ and $v_j's$ are multiple of q  $i.e$  $v_j = q.j$, $1 \leq j \leq p-1$. Now, Let $\pi = {\lbrace R,B \rbrace}$ be a bipartition of vertices in $L(\Gamma(\mathbb{Z}_n))$. Now we shall prove that $\pi$ is very cost effective bipartition. Since each set $R$ and $B$ contain $\frac{(p-1)(q-1)}{2}$ vertices. Therefore, set $R$ contains $\lbrace[u_1, v_k], [u_2, v_L], [u_3, v_k], [u_4, v_L],...,[u_{q-2}, v_k], [u_{q-1}, v_L]\rbrace$ where $1 \leq k \leq \frac{p-1}{2}$ and $\frac{p+1}{2} \leq L \leq p-1$ and $B$ contains $\lbrace[u_1, v_L], [u_2, v_k], [u_3, v_L], [u_4, v_k],...,$ $[u_{q-2}, v_L], [u_{q-1}, v_k]\rbrace$ where $1 \leq k \leq \frac{p-1}{2}$ and $\frac{p+1}{2} \leq L \leq p-1$. Now, take $[u_i, v_j] \in R$. Then $\vert N([u_i, v_j]) \cap R \vert = \frac{p+q}{2}-3$ and $\vert N([u_i, v_j]) \cap B \vert = \frac{p+q}{2}-1$. Therefore, $\vert N([u_i, v_j]) \cap R \vert < \vert N([u_i, v_j]) \cap B \vert$. Hence, $R$ is very cost effective set. Similarly, if we take $[u_i, v_j] \in B$, then $\vert N([u_i, v_j]) \cap B \vert < \vert N([u_i, v_j]) \cap R \vert$. Thus, the line graph $L(\Gamma(\mathbb{Z}_n))$ is very cost effective graph.

\end{proof}

\section{Bipartition in $N(\mathbb{Z}_n)$ and $\Omega(\mathbb{Z}_n)$}

In 2008, Bishop et al. \cite{bishop} introduced the concept of Nilradical graph and Non-Nilradical graph and further some work appeared in \cite{prakash}. They defined these graphs as follows:

\begin{definition}\cite{bishop}  The nilradical graph, denoted $N(R)$, is the graph whose vertices are the nonzero nilpotent elements of $R$ and two vertices are connected by an edge if and only if their product is $0$.
\end{definition}

\begin{definition}\cite{bishop}  The non-nilradical graph, denoted $\Omega(R)$, is the graph whose vertices are the non-nilpotent zero-divisors of $R$ and where two vertices are connected by an edge if and only if their product is $0$.
\end{definition}
  Fig 1 is serve an example of $N(\mathbb{Z}_{16})$. Also the example of $\Omega(\mathbb{Z}_18)$ is given below:

 \[\begin{tikzpicture}
	\vertex (4) at (-0.2,1) [label=left:$4$]{};
	\vertex (2) at (0.75,2.75) [label=left:$2$]{};
	\vertex (8) at (0.75,-0.75) [label=below:$8$]{};
	\vertex (9) at (2,1) [label=above:$9$]{};
	\vertex (16) at (3.25,2.75) [label=above:$16$]{};
	\vertex (10) at (3.25,-0.75) [label=right:$10$]{};
	\vertex (14) at (4.2,1) [label=right:$14$]{};
	\vertex (3)  at (0.25,1.75) [label=left:$3$]{};
	\vertex (15)  at (3.85,1.75) [label=right:$15$]{};
	\path
		(2) edge (9)
		(4) edge (9)
		(8) edge (9)
		(9) edge (10)
		(9) edge (14)
		(9) edge (16)

	;
\end{tikzpicture}\]
\hspace{4.5cm} \textbf{Figure 3}  $\Omega(\mathbb{Z}_{18})$

\begin{theorem}
Let $p, q$ be distinct primes and $n$, a positive integer. \\
   $(i)$ If $n = p^2$ and $p > 2$, then $N(\mathbb{Z}_n)$ is very cost effective graph.\\
   $(ii)$ If $n = p^2 q^2$ and $p, q \geq 2$, then $N(\mathbb{Z}_n)$ is very cost effective graph.\\
   $(iii)$ If $n = p^3$ and $p \geq 2$, then $N(\mathbb{Z}_n)$ is very cost effective graph.\\
   $(iv)$ If $n = p^2 q$ and $p > 2$, then $N(\mathbb{Z}_n)$ is very cost effective graph.\\
\end{theorem}

\begin{proof}

 $(i)$ Let $n = p^2$, where $p$ is a prime number and $p > 2$. Then the nilpotent elements in $\mathbb{Z}_n$ are $p, 2p,...,(p-1)p$. So, the number of nilpotent elements is $p-1$ and every element is adjacent to the other element. So, these $(p-1)$ elements forms a complete graph. Since $p$ is prime, then $p-1$ is even and every complete graph of even order is very cost effective. Thus, $N(\mathbb{Z}_n)$ is very cost effective graph.\\

   $(ii)$ Let $n = p^2 q^2$, where $p, q$ are distinct primes. Then the nilpotent elements in $\mathbb{Z}_n$ are multiple of $pq$ and number of nilpotent elements are $pq-1$. Since all the nilpotent elements are multiple of $pq$, so every vertex is adjacent to the all other vertices in $N(\mathbb{Z}_n)$. Therefore, $(pq-1)$ elements form a complete graph with $pq-1$ vertices. Since $p$ and $q$ are prime number and $p, q > 2$, so $pq-1$ is even number. Hence, the graph of $N(\mathbb{Z}_n)$ is very cost effective. \\

   $(iii)$ If $n = p^3$, where $p$ is prime number, then all the nilpotent elements are multiple of $p$ and total number of nilpotent elements are $p^2-1$. Now, we take bipartition $\pi = {\lbrace R,B \rbrace}$ of vertex set in $N(\mathbb{Z}_n)$ such that $R$ contains only those elements which are divisible by $p$ but not by $p^2$ and the set $B$ contains those elements which are divisible by only $p^2$. Then the set $R$ has $p(p-1)$ elements and the set $B$ has $p-1$ elements. Now, set $R$ is independent set and in $B$, each vertex is adjacent to every other vertices in $B$. Since all the elements of $R$ is adjacent to all elements of set $B$. Therefore, $R$ is very cost effective set. Now, take $u \in B$, then $\vert N(u) \cap B\vert = p-2$ and $\vert N(u) \cap R\vert = p(p-1)$. Since $\vert N(u) \cap B\vert < \vert N(u) \cap R\vert$, so set $B$ is very cost effective. Hence, $\pi = {\lbrace R,B \rbrace}$ is very cost effective bipartition and graph $N(\mathbb{Z}_n)$ is very cost effective. \\

   $(iv)$ Let $n = p^2 q$, where $p$ and $q$ are distinct prime number and $p \neq q$. Then the nilpotent elements of $N(\mathbb{Z}_n)$ are multiples of $pq$ and the number of nilpotent elements are $p-1$. These $p-1$ elements are connected to each other so these $p-1$ vertices forms a complete graph. Since $p$ is an odd prime so $p-1$ is even and complete graph of even number is very cost effective. Hence, $N(\mathbb{Z}_n)$ is very cost effective graph.

\end{proof}

\begin{theorem}
If $p$ and $q$ are distinct prime number and $n$ is a positive integer, then  $\Omega(\mathbb{Z}_n)$ is not very cost effective graph, where $n = p^2 q$.\\

\end{theorem}

\begin{proof}

 Let $n = p^2 q$, where $p$ and $q$ be a distinct primes. Then the non-nilradical elements are all zero-divisors which are not divisible by $pq$. Since these elements are not adjacent to themselves but the vertices which are multiple of $p^2$ is adjacent to multiple of $q$. Here, $p$ is also another vertex which is not adjacent to any other vertices. Therefore, $p$ is an isolated vertices. Hence $\Omega(\mathbb{Z}_n)$ is not very cost effective graph.

\end{proof}

\begin{theorem}

Let $p_1,p_2,...,p_m$ be distinct prime and $n = p_1.p_2...p_m, m \geq 1$. Then $\Omega(\mathbb{Z}_n)$ is very cost effective graph.\\

\end{theorem}

\begin{proof}
 Here, $\Omega(\mathbb{Z}_n) = \Gamma(\mathbb{Z}_n)$. So $\Omega(\mathbb{Z}_n)$ is very cost effective graph.

\end{proof}

\section{Bipartition in $T(\Gamma(\mathbb{Z}_n))$}

In this section, we have studied the very cost effective properties of $T(\Gamma(\mathbb{Z}_n))$ for $n = 2p$ and $n = pq$.

\begin{theorem}

Let $n = 2p$, p be an odd prime. Then the total graph $T(\Gamma(\mathbb{Z}_n))$ is not very cost effective.

\end{theorem}

\begin{proof}

 If $n = 2p$, then $\Gamma(\mathbb{Z}_n)$ is a star graph with $p$ vertices. Now, total graph of $\Gamma(\mathbb{Z}_n)$ is a graph with $2p-1$ vertices in which $p-1$ vertices have degree two, $p-1$ vertices have degree $p$ and one vertex which is itself $p$ has degree $2p-2$ respectively. In order to prove $T(\Gamma(\mathbb{Z}_n))$ is not very cost effective, let $T(\Gamma(\mathbb{Z}_n))$ be a very cost effective graph. Then it has a very cost effective bipartition $\pi = {\lbrace R,B \rbrace}$. Take $u \in R$ and suppose $u$ is a vertex whose degree is $2$. Then $\vert N(u) \cap R\vert = 0$ and $\vert N(u) \cap B\vert = 2$. Now, the vertex $p \in B$ because $u$ is adjacent to the vertex $p$. So, all the vertices whose degree are $2$ belong to the set $R$ because these vertices are adjacent to vertex $p$. Also, rest of the $p-1$ vertices which are of the form $[2, p], [2.2, p],...,[2(p-1), p]$ will be in $B$. If any one of these vertices belongs to $R$, then $R$ is not very cost effective set. So, take $[u_i, p] \in B$, where $u_i=2.i$, $1 \leq i \leq p-1$. Then $\vert N([u_i, p]) \cap B \vert = p-1$ and $\vert N([u_i, p]) \cap R \vert = 1$. Hence, the set $B$ is not very cost effective set, which contradicts our assumption. Thus, the graph $T(\Gamma(\mathbb{Z}_n))$ is not very cost effective.

\end{proof}

\begin{theorem}

Let $n = pq$, $p, q \geq 3$ and $p < q$ are primes. Then total graph $T(\Gamma(\mathbb{Z}_n))$ is very cost effective.
\end{theorem}

\begin{proof}

If $n= pq$, then $\Gamma({Z}_n)$ is a bipartite graph with $p+q-2$ vertices and $(p-1)(q-1)$ edges. Now, in total graph $T(\Gamma(\mathbb{Z}_n))$, there is $pq-1$ vertices in which $q-1$ vertices are multiples of $p$ and each has degree $2(p-1)$. Also $p-1$ vertices which are multiples of $q$ have degree $2(q-1)$. The vertices $[u_i, v_j] \in V(T(\Gamma(\mathbb{Z}_n))$,  $u_i = p.i$, $1\leq i\leq q-1$ and $v_j = q.j$, $1 \leq j \leq p-1$ have degree $p+q-2$. Suppose set $R$ contains the element $\lbrace p,2p,...,(q-1)p,[u_1, v_k],[u_2, v_L],[u_3, v_k],...,[u_{q-1}, v_L] \rbrace$ and $B$ contains $\lbrace q,2q,...,(p-1)q,[u_1, v_L],[u_2, v_k],[u_3, v_L],...,[u_{q-1}, v_k] \rbrace$ where $1 \leq k \leq \frac{p-1}{2}$,  $\frac{p+1}{2} \leq L \leq p-1$. Take bipartition $\pi = {\lbrace R,B \rbrace}$. We have to show that this bipartition is very cost effective.

   Now, consider $u_i \in R$, then $\vert N(u_i) \cap R \vert = p-2$ and $\vert N(u_i) \cap B \vert = p$. Also, for $[u_i, v_j] \in R$, we have $\vert N([u_i, v_j]) \cap R \vert = \frac{p+q-4}{2}$ and $\vert N([u_i, v_j]) \cap B \vert = \frac{p+q}{2}$. So, for every vertex $v \in R$, $\vert N(v) \cap R \vert < \vert N(v) \cap B \vert$ and the set $R$ is very cost effective set. In the set $B$, take $v_j \in B$, then $\vert N(v_j) \cap B \vert = q-2$ and $\vert N(v_j) \cap R \vert = q$. Again, take $[u_i, v_j] \in B$, we have $\vert N([u_i, v_j]) \cap B \vert = \frac{p+q-4}{2}$ and $\vert N([u_i, v_j]) \cap R \vert = \frac{p+q}{2}$. Therefore, $\vert N(v) \cap B \vert < \vert N(v) \cap R \vert$ for every vertex $v$ in $B$. So set $B$ is also very cost effective. Hence, the bipartition $\pi$ is very cost effective bipartition and the total graph $T(\Gamma(\mathbb{Z}_n))$ is very cost effective graph.

\end{proof}

\bibliographystyle{amsplain}

\end{document}